\documentclass[11pt]{article}

\usepackage[T2A]{fontenc}
\usepackage[cp1251]{inputenc}
\usepackage{amsthm}%%%%%%%% theorem definition
\usepackage{amsfonts}
\usepackage{eufrak}
\usepackage{amssymb}
\usepackage{amsmath}
\usepackage{cite}%%%%% citations

\usepackage{bm}
 \usepackage{amsthm}
\begin{document}
\renewcommand*{\proofname}{\textbf{\textit Proof}}
%\renewcommand*{\proofname}{\textbf{\upshape{\textit Proof}}}
%%%%%%%%%%%%% begin theorem definition %%%%%%%%%%%%%%%%%%
\newtheoremstyle{mytheorem}
  {\topsep}   % ABOVESPACE
  {\topsep}   % BELOWSPACE
  {\itshape}  % BODYFONT
  {}       % INDENT (empty value is the same as 0pt)
  {\bfseries} % HEADFONT
  { }         % HEADPUNCT
  {5pt plus 1pt minus 1pt} % HEADSPACE
  { }          % CUSTOM-HEAD-SPEC
\newtheoremstyle{myremark}
  {\topsep}   % ABOVESPACE
  {\topsep}   % BELOWSPACE
  {\upshape}  % BODYFONT
  {}       % INDENT (empty value is the same as 0pt)
  {\bfseries} % HEADFONT
  {  }         % HEADPUNCT
  {5pt plus 0pt minus 1pt} % HEADSPACE
  { }          % CUSTOM-HEAD-SPEC\cite{}
\theoremstyle{mytheorem}
\newtheorem{theorem}{Theorem}[section]
 \newtheorem{theorema}{Theorem}

 \newtheorem*{heyde*}{The Heyde theorem}
 \newtheorem*{a*}{Theorem A}
 \newtheorem*{b*}{Theorem B}
 \newtheorem{conjecture}[theorem]{Conjecture}

 \newtheorem{proposition}[theorem]{Proposition}
 \newtheorem{lemma}[theorem]{Lemma}
\newtheorem{corollary}[theorem]{Corollary}
\newtheorem{definition}{Definition}[section]
\theoremstyle{myremark}
\newtheorem{remark}[theorem]{Remark}
%%%%%%%%%%%%%%%%%%%%% end theorem definition %%%%%%%%%%%%%%%%%%
\noindent This article is accepted for publishing in

\noindent the  Journal of Fourier and Analysis and Applications

\vskip 1 cm

\noindent{\textbf{Heyde theorem for locally compact Abelian groups containing no subgroups  }}

\noindent{\textbf{topologically isomorphic  to the 2-dimensional torus
  }}
\bigskip

\noindent{{Gennadiy Feldman}} 
(ORCID ID  https://orcid.org/0000-0001-5163-4079)

\bigskip

\noindent\textbf{Abstract} \\ We prove the following group analogue of the 
well-known Heyde theorem on a characterization of the Gaussian distribution 
on the real line. Let $X$ be a  second countable locally compact   Abelian  group containing  no subgroups topologically isomorphic to the 2-dimensional torus.
Let $G$ be the subgroup
of $X$ generated by all elements of $X$ of order $2$  and let
$\alpha$ be a topological automorphism of the group  
$X$ such that   ${\rm Ker}(I+\alpha)=\{0\}$.
Let $\xi_1$ and $\xi_2$ be independent random
variables with  values in $X$ and distributions    $\mu_1$ and $\mu_2$ with nonvanishing characteristic functions. If the conditional
distribution of the linear form $L_2 = \xi_1 + \alpha\xi_2$ given $L_1 = \xi_1 +
\xi_2$ is symmetric, then  $\mu_j$ are convolutions of Gaussian distributions 
on $X$ and distributions supported in $G$. We also prove  that this 
theorem is false if $X$ is 
the 2-dimensional torus.

\bigskip


\bigskip


\section{Introduction}

The following characterization theorem was proved in  \cite[Theorem 3.1]{Fe?}.
\begin{a*}\label{A}
Let $X$ be a second countable locally compact    Abelian group with the connected component of zero of dimension $1$. Let $G$ be the subgroup of $X$
generated by all elements of $X$ of order $2$  and let $\alpha$ be a topological automorphism of the group    $X$ satisfying the condition 
\begin{equation}\label{d1}
{\rm Ker}(I+\alpha)=\{0\}.
\end{equation}
 Let $\xi_1$ and $\xi_2$ be independent random variables with   values in   $X$ and distributions
$\mu_1$ and $\mu_2$ with nonvanishing characteristic functions.
If the conditional distribution of the linear form    $L_2 = \xi_1 + \alpha\xi_2$ given $L_1 = \xi_1 + \xi_2$ is symmetric, then
$\mu_j$ are convolutions of Gaussian distributions 
on $X$ and distributions supported in $G$.
\end{a*}
Theorem A can be considered as a group analogue for two independent random variables of the well-known   theorem  of C.C.~Heyde   (\!\!\cite{He}, see also \cite[\S\,13.4.1]{KaLiRa}), where
the Gaussian distribution on the real line is characterized by the symmetry of the conditional distribution of one linear form of independent random variables given another.

Denote by  $\mathbb{R}$ the additive group of real numbers and by  $\mathbb{T}$  the circle group (the one-dimensional torus), i.e., the multiplicative group of all complex numbers with absolute value 1. The aim of the article is, firstly, to prove that Theorem A is true for a much wider class of locally compact Abelian groups, namely  for second
countable locally compact Abelian groups containing no  subgroups topologically
isomorphic to the $2$-dimensional torus $\mathbb{T}^2$ (Theorem \ref{th1}).
Secondly, to prove that Theorem A is false when $X$ is the $2$-dimensional torus
$\mathbb{T}^2$ (Theorem \ref{th2}). 
We emphasize that the proof of Theorem \ref{th1} given in the article is fundamentally different from the proof of Theorem A, which is based on the proof of the theorem for the group $X$ of the form $X=\mathbb{R}\times D$, where $D$ is a second countable totally disconnected locally compact Abelian group, and uses complex analysis.    

Many studies have been devoted to analogues of Heyde's theorem for different 
classes of locally compact Abelian groups (see e.g.
\cite{Fe2, Fe4, Fe3, Fe20bb, Fe6, F, Fe7, {Fe8a}, Fe2020, My2, My1, MyF, JFAA2021}, and also   \cite[Chapter IV]{book2023}, where one can find additional references). 
 In the article we continue  this research.

We use in the article standard facts related to abstract harmonic analysis, see e.g. \cite{Hewitt-Ross}. Let $X$ be a second
countable locally compact Abelian group,  let   ${\rm Aut}(X)$ be the  group of topological
automorphisms of the group $X$, and  let  $I$ be the identity automorphism of 
a group. Denote by $Y=X^*$ the character group of the group    $X$, and by  $(x,y)$ the value of a character $y \in Y$ at an
element $x \in X$. For a closed subgroup $K$  of the group $X$,  
denote by $A(Y, K) = \{y \in Y: (x, y) =1$ for all $x\in
K\}$ its annihilator. The character group of the factor-group $X/K$ 
is topologically isomorphic to 
the annihilator $A(Y, K)$. Let  $\alpha:X\rightarrow X$ be 
a continuous endomorphism of the group $X$.
The adjoint endomorphism $\widetilde\alpha: Y\rightarrow Y$   is defined
as follows:
    $(\alpha x,
y)=(x, \widetilde\alpha y)$ for all $x\in X$, $y\in
Y$.  Let $n$ be a natural number. Put $X^{(n)}=\{nx: x\in X\}$.
A topological isomorphism of 
  locally compact Abelian groups $X_1$ and $X_2$ is denoted as  
$X_1\cong X_2$.  Denote by $\mathbb{Z}$ the additive group of integers.  

Let   $f(y)$
be a  function on the group $Y$ and let $h\in Y$. Denote by $\Delta_h$
the finite difference operator
$$\Delta_h f(y)=f(y+h)-f(y), \quad y\in Y.$$

Denote by ${\rm M}^1(X)$ the
convolution semigroup of probability distributions on the group $X$.
Let $\mu\in{\rm M}^1(X)$. Denote by
$$
\hat\mu(y) =
\int_{X}(x, y)d \mu(x), \quad y\in Y,$$   the characteristic function (Fourier transform) of the distribution   $\mu$. The   
characteristic 
function of a signed measure is defined in the same way.   
Denote by $\sigma(\mu)$ the support of $\mu$. Define the distribution $\bar \mu \in {\rm M}^1(X)$ by the formula
 $\bar \mu(B) = \mu(-B)$ for any Borel subset $B$ in $X$.
Then $\hat{\bar{\mu}}(y)=\overline{\hat\mu(y)}$.  If $F$ is a Borel subgroup of $X$, denote by ${\rm M}^1(F)$ the subsemigroup of ${\rm M}^1(X)$  consisting of all distributions concentrated on $F$.

A distribution $\gamma \in {\rm M^1}(X)$ is
called Gaussian (\!\!\cite[Chapter IV, \S 6]{Pa}) if its characteristic
function can be represented in the form
\begin{equation}\label{01.05.1}
\hat\gamma(y)= (x,y)\exp\{-\varphi(y)\}, \quad  y\in Y,
\end{equation}
where $x \in X$ and $\varphi(y)$ is a continuous nonnegative
function on the group $Y$ satisfying  the equation
\begin{equation}\label{09.07.1}
\varphi(u + v) + \varphi(u
- v) = 2[\varphi(u) + \varphi(v)], \quad u,  v \in
Y.
\end{equation}
 A Gaussian distribution is called symmetric if  $x=0$ in   (\ref{01.05.1}). Denote by $\Gamma(X)$ the set of Gaussian distributions on the
group $X.$ Note that in particular, the degenerate  distributions are Gaussian. Let $x\in X$. Denote by $E_x$
the degenerate distribution concentrated at the point $x\in X$. 

Denote by $m_X$ a Haar measure on the group $X$. It is well known that 
$m_X$ is unique up to a positive
multiplicative constant. If $X$ is a compact group, then
$m_X(X)<\infty$. We suppose that in this case $m_X\in {\rm M^1}(X)$.
 Note that, if  
$X$ is an arbitrary locally compact Abelian group and 
 $K$ is a compact subgroup of  $X$, then
the characteristic function $\widehat m_K(y)$ is of the form
\begin{equation}\label{10.1}
\widehat m_K(y)=
\begin{cases}
1, & \text{\ if\ }\ \ y\in A(Y, K),
\\ 0, & \text{\ if\ }\ \ y\not\in A(Y, K).
\end{cases}
\end{equation} 

\section{Proof of the main theorem} 

The main result of the article is the following theorem.

\begin{theorem}  \label{th1}
  Let $X$ be a second
countable locally compact Abelian group containing no  subgroups topologically
isomorphic to the $2$-dimensional torus $\mathbb{T}^2$. Let $G$ be the subgroup of $X$
generated by all elements of $X$ of order $2$  and let $\alpha$ be a topological automorphism of the group    $X$ satisfying   condition $(\ref{d1})$. Let
  $\xi_1$ and  $\xi_2$ be independent random variables with values in the group
       $X$  and distributions $\mu_1$ and $\mu_2$ with nonvanishing characteristic functions.
   If the conditional distribution of the linear form    $L_2 = \xi_1 + \alpha\xi_2$ given $L_1 = \xi_1 + \xi_2$ is symmetric, then
$\mu_j\in \Gamma(X)*{\rm M^1}(G)$, $j=1, 2$.
\end{theorem}

To prove Theorem \ref{th1} we need some lemmas.
\begin{lemma}\label{lemM} {\rm (\!\!\cite[Lemma 3.8]{My1}, see also \cite[Corollary 9.7]{book2023})}    Let  $X$ be a second
countable locally compact Abelian group and let $\alpha$
be a topological automorphism of $X$.
Let  $\xi_1$ and  $\xi_2$ be independent random variables with values in
the group $X$.
   If the conditional distribution of the linear form    $L_2 = \xi_1 +
\alpha\xi_2$  given $L_1 = \xi_1 +
\xi_2$  is symmetric,   then the linear forms
$P_1=(I+\alpha)\xi_1+2\alpha\xi_2$ and
$P_2=2\xi_1+(I+\alpha)\xi_2$  are independent.
\end{lemma}
\begin{lemma}\label{lem13_1} {\rm (\!\!\cite[Lemma 10.1]{Fe5})}   Let $X$ be a second
countable locally compact Abelian group with character group $Y$. Let
$\alpha_j$, $\beta_j$, $j=1, 2$, be continuous endomorphisms of the group $X$. Let $\xi_1$ and $\xi_2$ be independent random variables with values in the group
       $X$  and distributions
 $\mu_1$ and $\mu_2$.
 The linear forms $L_1 = \alpha_1\xi_1 +
\alpha_2\xi_2$ and $L_2 = \beta_1\xi_1  + \beta_2\xi_2$ are independent
if and only if the characteristic functions $\hat\mu_j(y)$ satisfy the
equation
\begin{equation}\label{4}
\hat\mu_1(\widetilde\alpha_1 u+\widetilde\beta_1 v)\hat\mu_2(\widetilde\alpha_2 u+\widetilde\beta_2 v)=\hat\mu_1(\widetilde\alpha_1
u)\hat\mu_2(\widetilde\alpha_2
u)\hat\mu_1(\widetilde\beta_1 v)\hat\mu_2(\widetilde\beta_2 v),
\quad u, v \in
Y.
\end{equation}
\end{lemma}
\begin{lemma}\label{lem1}{\rm (\!\!\cite[Lemma 6]{Fe20bb}, see also \cite[Lemma 5.1]{book2023})}   
Let
$Y$ be a locally compact Abelian group  and  let 
$A(y)$ be a continuous function on the group $Y$ satisfying the equation
$$\Delta_{2k}\Delta^2_{h} A(y) = 0, \quad y, k, h  \in Y,
$$
and the conditions $A(-y)=A(y)$, $A(0)=0$. Let 
\begin{equation}\label{09.07.3}
Y = \bigcup_\iota {(y_\iota +
\overline{Y^{(2)}})}, \quad y_0=0,
\end{equation}
be  a $\overline{Y^{(2)}}$-coset decomposition of the group $Y$.
Then the function
$A(y)$ can be represented  in the form
\begin{equation}\label{09.07.2}
A(y) = \varphi(y) +
r_\iota, \quad y \in y_\iota + \overline{Y^{(2)}},
\end{equation}
where $\varphi(y)$ is a continuous function on the group    $Y$ satisfying equation $(\ref{09.07.1})$.
\end{lemma}
Let us recall some definitions. Let $X$ be an Abelian group. We do not assume that $X$ is a topological group. If each element of $X$ has finite order, then we say that $X$ is a torsion group. We say that $X$ is a torsion-free group if     
each element of $X$, except zero, has infinite order. The subgroup consisting 
of all elements of finite order of the group $X$ is called a torsion part of $X$.

We formulate as a lemma the following corollary of the well-known Baer--Fomin theorem (\!\!\cite[Theorem 100.1]{Fu2}).
\begin{lemma}\label{nlem4} Let $Y$ be an
   Abelian group and let $H$ be a torsion part of $Y$. If $H$ is a bounded subgroup, 
   then $H$ is a direct factor of $Y$.  
\end{lemma}
Using Lemma \ref{nlem4} we shall prove the following statement.
\begin{lemma}\label{lem4} Let $X$ be a
  locally compact Abelian group containing no  subgroups topologically
isomorphic to the circle group $\mathbb{T}$. Let   $K$ be a closed bounded subgroup of $X$. Then the factor-group $X/K$ also contains no  subgroups topologically
isomorphic to the circle group $\mathbb{T}$.
\end{lemma}
\begin{proof}  The lemma is true if $K=X$, so we suppose that $K\ne X$. Assume that the factor-group $X/K$   contains a subgroup $F$ such that $F\cong \mathbb{T}$. Let us prove that in this case we arrive at a contradiction.
Let $p:X \rightarrow X/K$ be the natural homomorphism.  Put $S=p^{-1}(F)$. 
The group $S$ is closed and hence locally compact.
Consider the restriction of $p$ to the subgroup $S$. We have $p:S \rightarrow S/K$ and $S/K\cong\mathbb{T}$. Thus, we can  prove the lemma assuming that  $X/K\cong \mathbb{T}$.
By the structure theorem for 
locally compact Abelian groups,  the group  $X$ is topologically isomorphic to a group of the form $\mathbb{R}^m\times G$, where a locally compact Abelian group $G$ contains a compact open subgroup. Since $K$ is a  bounded subgroup, $K\subset G$. It follows from $X/K\cong \mathbb{T}$ that $m=0$, i.e., the group $X$ itself contains a compact open subgroup. Denote by $Y$ the character group of the group $X$. 
  
  First suppose that $X$ is a compact group. Then $Y$ is a discrete group. Denote by $H$ the torsion part 
  of the group $Y$.
We have  $A(Y, K)\cong (X/K)^*$. Since $X/K\cong \mathbb{T}$, this implies that 
\begin{equation}\label{new1}
A(Y, K)\cong \mathbb{Z}.
\end{equation}
   Inasmuch as  $K$ is a bounded group, there is a natural $n$ 
   such that $nx=0$ for all $x\in K$. It follows from this that $ny\in A(Y, K)$ for each $y\in Y$, i.e., 
\begin{equation}\label{new4}
Y^{(n)}\subset A(Y, K).
\end{equation}   
  In particular,  
\begin{equation}\label{new2}
H^{(n)}\subset A(Y, K).
\end{equation}   
In view of   
(\ref{new1}), (\ref{new2}) implies that
\begin{equation}\label{new3}
H^{(n)}=\{0\}.
\end{equation}
Hence 
 $H$ is a bounded group. By Lemma \ref{nlem4},
 $H$ is a direct factor of  $Y$. We have $Y= H\times L$, where $L$ is a torsion-free group. Taking into account (\ref{new3}), this implies that  $Y^{(n)}= L^{(n)}$. For this reason, it follows from (\ref{new1}) and (\ref{new4}) that 
 $L\cong \mathbb{Z}$. This implies that the group $X$ contains a subgroup topologically isomorphic to the circle group $\mathbb{T}$, contrary to the assumption. Thus, the lemma is proved if the group $X$ is compact.

Consider the general case. Let $B$ be a compact open subgroup of the group $X$.  As far as  $p$ is a continuous open epimorphism, $p(B)$ is an open subgroup of the factor-group $X/K$. Since $X/K\cong \mathbb{T},$ we have $p(B)=X/K.$ Moreover, $p(B)\cong B/(K\cap B)$. Thus, $B/(K\cap B)\cong \mathbb{T}$. As has been shown above, this implies that the group $B$, and hence the group $X$, contains a subgroup topologically isomorphic to the circle group  $\mathbb{T}$, contrary to the assumption.     \end{proof}

Let $\mathbb{R}^{\aleph_0}$ be the space of all sequences of real numbers 
 considering in the product topology. We will need the definition of the Gaussian distribution in $\mathbb{R}^{\aleph_0}$. The space $\mathbb{R}^{\aleph_0}$
 is one of the simplest examples of a locally convex space. We note that 
 Gaussian distributions
in arbitrary locally convex spaces are studied in details in the 
fundamental monograph \cite{B}. 

  Denote by $\mathbb{R}^{\aleph_0*}$ the space of all finitary 
sequences of real numbers with the topology of  strictly inductive 
limit of spaces $\mathbb{R}^{n}$. Let ${t}=(t_1,t_2, \dots, t_n, \dots)\in
\mathbb{R}^{\aleph_0}$ and ${s}=(s_1, s_2, \dots, s_n, 0, \dots)\in
\mathbb{R}^{\aleph_0*}$. Set
$$ 
\langle{t},{s}\rangle=\mathop{\sum}\limits_{j=1}^\infty t_js_j 
.$$ 
Let $\mu$ be a
distribution on $\mathbb{R}^{\aleph_0}$. We define the
characteristic function of  $\mu$ by the formula
$$
\hat\mu({s})=\int\limits_{\mathbb{R}^{\aleph_0}}
\exp\{i\langle{t},{s}\rangle\}d\mu({t}), 
\quad {s}\in \mathbb{R}^{\aleph_0*}.
$$
Let  $A=(\alpha_{ij})_{i, j=1}^\infty$ be a symmetric positive
semidefinite matrix, i.e.,
the quadratic form
$$\langle As, {s}\rangle=\mathop{\sum}\limits_{i, j=1}^\infty 
\alpha_{ij}s_is_j$$
is nonnegative for all   ${s}=(s_1,s_2, \dots, s_n, 0, \dots)\in
\mathbb{R}^{\aleph_0*}$.

A distribution $\mu$ on the group $\mathbb{R}^{\aleph_0}$ is
called 	Gaussian  if its characteristic function
is represented in the form 	$$
\hat\mu({s})=\exp\{i\langle{t},{s}\rangle-\langle A {s},
	{s}\rangle\}, \quad {s}\in \mathbb{R}^{\aleph_0*},
$$
	where ${t}\in \mathbb{R}^{\aleph_0}$  and $A=(\alpha_{ij})_{i,
		j=1}^\infty$ is a symmetric positive
	semidefinite matrix.

\begin{lemma}\label{lem3} {\rm (\!\!\cite{Fe1}, see also \cite[\S 3]{Fe5})}  Let $X$ be a second countable
  locally compact Abelian group containing no  subgroups topologically
isomorphic to the circle group $\mathbb{T}.$ Then there exists a 
continuous monomorphism $p:E\rightarrow X$,
 where either  $E=\mathbb{R}^{n}$ for some $n$ or
$E=\mathbb{R}^{\aleph_0}$, such that if $\gamma$ is a symmetric
Gaussian distribution on  $X$, then $\gamma=p(M)$, where $M$ is   a symmetric Gaussian distribution  on $E$.
\end{lemma}
\begin{lemma}\label{lem5} Let $X$ be  a second countable locally compact Abelian group containing no  subgroups topologically
isomorphic to the circle group $\mathbb{T}.$ Let $G$ be the subgroup of $X$
generated by all elements of $X$ of order $2$.   Let  $\mu\in\Gamma(X)*{\rm M}^1(G)$ and suppose that the characteristic function of the distribution $\mu$
 does not vanish.
If $\mu=\mu_1*\mu_2$,  where $\mu_j\in {\rm M}^1(X)$, then $\mu_j\in \Gamma(X)*{\rm M}^1(G)$, $j=1, 2$.
 \end{lemma}
\begin{proof} Since the group  $X$ contains  no  subgroups topologically
isomorphic to the circle group $\mathbb{T}$, let $p$ be 
a continuous monomorphism $p:E\rightarrow X$ which exists by 
Lemma \ref{lem3}. As far as $p$ is a monomorphism, we have $p(E) \cap G =
\{0\}$. Hence  $p$ can be extended to a continuous monomorphism
 $\bar p: E \times G
\rightarrow X$ by the formula $\bar p(t, g) = p(t) + g$, $t
\in E$, $g \in G$, and $\bar p$ generates an isomorphism of the semigroups  ${\rm M}^1(E\times G)$ and ${\rm M}^1(p(E)\times
G)$.   Let $\mu=\gamma*\omega$, where $\gamma\in\Gamma(X)$, $\omega\in {\rm M}^1(G)$. We can assume without loss of generality that $\gamma$ is a symmetric Gaussian distribution. By Lemma \ref{lem3}, $\gamma=p(M)$, where $M\in \Gamma(E)$. Hence  
$\mu=\bar p(N)$, where $N=M*\omega$. Obviously, the distribution $\mu$ is 
concentrated on
the Borel subgroup $p(E)\times G$ of the group $X$. Substituting, if
it is necessary, the distributions $\mu_j$ by their shifts, we can
assume that $\mu_j$ are also concentrated on the Borel subgroup
$p(E)\times G$. Since the semigroups  ${\rm M}^1(E\times G)$ and ${\rm M}^1(p(E)\times
G)$ are isomorphic, we have $\mu_j = \bar p(N_j)$, where
$N_j \in
 {\rm M^1}(E \times G)$ and $N_j$ is a factor of
$N$.  By \cite[Lemma 5 and Remark 4]{Fe4}, we have $N_j=M_j*\omega_j$, 
where $M_j \in \Gamma(E)$,
$\omega_j \in {\rm M^1}(G)$. Hence  $$\mu_j=\bar p (N_j)=\bar
p(M_j*\omega_j)= \bar p(M_j)*\bar p(\omega_j)=p(M_j)*  \omega_j=\gamma_j*\omega_j,$$  where $\gamma_j=  p(M_j)$.
Since $\gamma_j \in \Gamma(X)$, the  lemma is proved. \end{proof}

\noindent\textbf{\textit{Proof of Theorem $\bf{\ref{th1}}$}}  Assume that a locally compact Abelian group $X$ 
 contains no subgroups topologically isomorphic to the 2-dimensional torus 
 $\mathbb{T}^2$  but  contains a subgroup topologically isomorphic to the circle 
 group  $\mathbb{T}$. Obviously, there can be only one such subgroup. Hence
it is  invariant with respect to each topological automorphism of the group  $X$. 
Since ${\rm Aut}(\mathbb{T})=\{\pm I\}$, this implies that there is no topological automorphism $\alpha$ of the group $X$ satisfying condition 
 $(\ref{d1})$.
  Hence  we can prove the theorem assuming from the beginning that the group $X$ contains no subgroups topologically isomorphic to the circle group  $\mathbb{T}$.

Denote by $Y$  the character group of the group $X$.
By Lemma \ref{lemM}, the symmetry of the conditional distribution of the linear form   $L_2$ given  $L_1$ implies that the linear forms
$P_1=(I+\alpha)\xi_1+2\alpha\xi_2$ and
$P_2=2\xi_1+(I+\alpha)\xi_2$ are independent. By Lemma  \ref{lem13_1}, 
it follows from this that the characteristic functions
$\hat\mu_j(y)$ satisfy   equation $(\ref{4})$,
which takes the form
\begin{multline}\label{5}
\hat\mu_1((I+\widetilde\alpha) u+2 v)\hat\mu_2(2\widetilde\alpha  u+(I+\widetilde\alpha) v)\\=\hat\mu_1((I+\widetilde\alpha)
u)\hat\mu_2(2\widetilde\alpha
u)\hat\mu_1(2 v)\hat\mu_2((I+\widetilde\alpha) v), \quad u, v \in
Y.
\end{multline}
Put $\nu_j = \mu_j* \bar \mu_j$. We have  $\hat \nu_j(y) = |\hat \mu_j(y)|^2 > 0$ 
for all   $y \in Y$, $j=1, 2$.
The characteristic functions
 $\hat \nu_j(y)$
also satisfy equation $(\ref{5})$. Set $\psi_j(y)=-\ln\hat\nu_j(y)$, $j=1, 2$. It follows from
$(\ref{5})$ that the functions $\psi_j(y)$ satisfy the equation
\begin{equation}\label{6}
\psi_1((I+\widetilde\alpha) u+2 v)+\psi_2(2\widetilde\alpha  u+(I+\widetilde\alpha) v)=A(u)+B(v), \quad u, v \in
Y,
\end{equation}
where
\begin{equation}\label{16.09.15.1}
A(y)=\psi_1((I+\widetilde\alpha)
y)+\psi_2(2\widetilde\alpha
y), \quad B(y)=\psi_1(2 y)+\psi_2((I+\widetilde\alpha) y), \quad y\in Y.
\end{equation}
Equation (\ref{6}) has already appeared  earlier in the study of Heyde's theorem on various locally compact Abelian groups (see e.g. \cite{Fe7}, \cite{Fe8a}). For completeness, we give here its solution. We use the finite difference method.
Take  an arbitrary element $h_1$  of the group $Y$.
Substitute  $u+(I+\widetilde\alpha) h_1$ for $u$ and $v-2\widetilde\alpha  h_1$ for $v$ in equation
 (\ref{6}).
 Subtracting equation
  (\ref{6}) from the obtaining equation  we get
  \begin{equation}\label{7}
    \Delta_{(I-\widetilde\alpha)^2 h_1}{\psi_1((I+\widetilde\alpha) u+2 v)}
    =\Delta_{(I+\widetilde\alpha) h_1} A(u)+\Delta_{-2\widetilde\alpha  h_1} B(v),
\quad u, v\in Y.
\end{equation}
Take an arbitrary element $h_2$ of the group $Y$.
Substitute 
  $u+2h_{2}$ for  $u$ and
$v-(I+\widetilde\alpha)h_{2}$ for $v$ in equation  (\ref{7}). Subtracting equation
  (\ref{7}) from the resulting equation  we obtain
 \begin{equation}\label{8}
     \Delta_{2 h_2}\Delta_{(I+\widetilde\alpha) h_1} 
     A(u)+\Delta_{-(I+\widetilde\alpha) h_2}\Delta_{-2\widetilde\alpha  h_1} B(v)=0,
\quad u, v\in Y.
\end{equation}
Take an arbitrary element $h$ of the group $Y$.
Substituting  $u+h$ for $u$ in equation
(\ref{8})   and subtracting equation
  (\ref{8}) from the obtaining equation  we get
 \begin{equation}\label{9}
   \Delta_{h}\Delta_{2 h_2}\Delta_{(I+\widetilde\alpha) h_1} A(u)=0,
\quad u\in Y.
\end{equation}

It follows from properties of adjoint homomorphisms that if
 a topological automorphism
$\alpha$ satisfies condition
$(\ref{d1})$, then the subgroup
$(I+\widetilde\alpha)(Y)$  is dense in
$Y$. Taking into account that
$h, h_1, h_2$ are arbitrary elements of the group  $Y$, it follows from
(\ref{9}) that the function $A(y)$ satisfies the equation 
\begin{equation}\label{10}
\Delta_{2k}\Delta_h^{2} A(y) = 0, \quad y, k, h  \in Y.
\end{equation}
By Lemma \ref{lem1},   (\ref{10}) implies that the function  $A(y)$ is represented in the form (\ref{09.07.2}), where the function $\varphi(y)$, 
as is easily seen, is nonnegative.
 Denote by $\mu$ the distribution of the random variable
$P_1=(I+\alpha)\xi_1+2\alpha\xi_2$. It is obvious that 
$$\mu=(I+\alpha)(\mu_1)*(2\alpha)(\mu_2).$$ 
 Put
$\nu=\mu*\bar\mu$.
  Then
\begin{equation}\label{11.07.1}
\nu=(I+\alpha)(\mu_1)*(2\alpha)(\mu_2)*
(I+\alpha)(\bar\mu_1)*(2\alpha)(\bar\mu_2).
\end{equation}
It follows from (\ref{16.09.15.1}) and  (\ref{11.07.1})
that the characteristic function
 $\hat\nu(y)$ is of the form
 \begin{equation}\label{05.05.1}
\hat\nu(y)=e^{-A(y)}, \quad y\in Y.
\end{equation}

Denote by $\gamma$ the Gaussian distribution on the group $X$ with the characteristic function
\begin{equation}
\label{e17}\hat \gamma(y) = \exp\{-\varphi(y)\}, \quad y \in
Y.
\end{equation}
Taking into account  (\ref{09.07.3}) and  (\ref{09.07.2}), define on the group $Y$ the function $g(y)$ by the formula
\begin{equation}
\label{e18}g(y) = \exp\{- r_\iota\}, \quad y \in y_\iota +
\overline{Y^{(2)}}.
\end{equation}
Since $g(y)=\hat\nu(y)/\hat\gamma(y)$, the function $g(y)$ is continuous. Moreover, (\ref{e18}) implies that the function   $g(y)$ is invariant with respect to the subgroup   $\overline{Y^{(2)}}$. Check that  $g(y)$
is a positive definite function. Note that
\begin{equation}
\label{18.12}A(X, \overline{Y^{(2)}})=G.
\end{equation} 
 Consider   decomposition
(\ref{09.07.3}) and take a finite set of elements   
$y_{\iota_j}$, $j = 1, 2,\dots ,n$.   Let  $H$ be a subgroup of $Y$ generated all cosets
$y_{\iota_j} + \overline{Y^{(2)}}$. Obviously, it suffices to verify that the restriction of the function $g(y)$ to $H$ is a positive definite function. The subgroup
  $H$ consists of a finite number of cosets   $y_\iota +
\overline{Y^{(2)}}$. Put
   $K = A(X, H)$. Then   $H^* \cong X/K$ and obviously, $K
\subset G$.
Consider the restriction of the function  $g(y)$ to $H$.
This restriction is invariant with respect to the subgroup
$\overline{Y^{(2)}}$  and hence defines a function on the factor-group
$H/\overline{Y^{(2)}}$. We note that
 $H/\overline {Y^{(2)}}$ is a finite group  and all its nonzero elements are of the order 2.  It follows from this that any real-valued function on the factor-group
$H/ \overline {Y^{(2)}}$  is the characteristic function of a signed measure.
In particular, the restriction of the function    $g(y)$ to the subgroup $H$ is the characteristic function of a signed measure $\varpi$. Since   $(H/\overline{Y^{(2)}})^*\cong A(X/K,  \overline {Y^{(2)}})$, we can consider the  signed measure
 $\varpi$ as a  signed measure on the finite subgroup   $F = A(X/K,  \overline {Y^{(2)}})$. It follows from
(\ref{09.07.2}) and (\ref{05.05.1})--(\ref{e18}) that
the restriction of the characteristic function $\hat\nu(y)$ to $H$ is the characteristic function of the convolution of a Gaussian distribution $\lambda$ on $X/K$ and the signed measure  $\varpi$ on $F$. We verify that the signed measure $\varpi$ is actually a distribution  and this proves that  $g(y)$
is a positive definite function.

Since $K
\subset G$, by Lemma \ref{lem4}, the factor-group $X/K$ contains no subgroups topologically isomorphic to the circle group   $\mathbb{T}$. Then by Lemma \ref{lem3}, applying to the group $X/K$,   there exists a continuous monomorphism $p:E\rightarrow X/K$,
 where either  $E=\mathbb{R}^{n}$ for some $n$ or
$E=\mathbb{R}^{\aleph_0}$, such that $\lambda=p(M)$, where $M$ is   a symmetric
Gaussian distribution  on
 $E$.
 Hence  the Gaussian distribution   $\lambda$ is concentrated on the Borel subgroup $p(E)$ of the group $X/K$. Note that all nonzero elements of the group  $F$ are of the order   2  and hence 
\begin{equation}
\label{e191} p(E)\cap F = \{0\}.
\end{equation}
Since the convolution $\lambda*\varpi$ is a distribution, in view of  (\ref{e191}), 
$\varpi$ is also a distribution. Thus, we proved that  $g(y)$ is a continuous 
  positive definite function such that $g(0)=0$. By the Bochner theorem, there exists a distribution  $\omega \in {\rm M}^1(X)$ such that $\hat\omega(y) =
g(y)$, $y \in
Y$. Since $g(y)= 1$ for $y\in \overline{Y^{(2)}}$, in view of (\ref{18.12}), we have an inclusion
   $$\sigma(\omega) \subset A\left(X, \{y\in Y: g(y)=1\}\right)
\subset A(X, \overline{Y^{(2)}})=G.$$  It follows from
$\hat\nu(y)=\hat\gamma(y)g(y)$ for all $y \in
Y$, that $\nu= \gamma*\omega\in\Gamma(X)*{\rm M}^1(G)$.  By Lemma \ref{lem5},   (\ref{11.07.1}) implies that
 $(I+\alpha)(\mu_1)\in \Gamma(X)*{\rm M}^1(G)$. Taking into account that $I+\alpha$ is a continuous monomorphism, we get $\mu_1\in \Gamma(X)*{\rm M}^1(G)$.

To complete the proof of the theorem, it remains to prove 
that $\mu_2\in \Gamma(X)*{\rm M}^1(G)$.  Take an arbitrary element  
$k$ of the group $Y$ and
substitute   $v+k$ for $u$ in
(\ref{8}). Subtracting equation
  (\ref{8}) from the resulting equation, we obtain
  $$
   \Delta_{k}\Delta_{-(I+\widetilde\alpha) h_2}\Delta_{-2\widetilde\alpha  h_1} B(v)=0,
\quad v\in Y.
$$
Consider the distribution of the random variable  $P_2=2\xi_1+(I+\alpha)\xi_2$. 
Arguing in the same way as in the case when we considered the distribution of the random variable $P_1$, we prove that
$\mu_2\in \Gamma(X)*{\rm M}^1(G)$. 

\hfill$\Box$

Evidently, if the connected component of zero of a second countable locally compact    Abelian group $X$ has  dimension $1$, then $X$ contains no subgroups topologically 
isomorphic to the 2-dimensional torus $\mathbb{T}^2$. For this reason Theorem A follows from Theorem  \ref{th1}.

We note that the statement of Theorem  \ref{th1} in the case 
if a topological automorphism $\alpha$ 
of the group $X$ satisfies the conditions
\begin{equation}
\label{fe18}I\pm\alpha\in {\rm Aut(X)}  
\end{equation}
follows from \cite[Theorem 1]{Fe20bb}.

It is obvious that if a locally compact Abelian group $X$ contains 
no elements of order  2, then $X$ contains no subgroups topologically 
isomorphic to the 2-dimensional torus $\mathbb{T}^2$. Hence Theorem \ref{th1} implies the following characterization of the Gaussian distribution proved earlier in  \cite[Theorem 3]{Fe7}, see also \cite[Theorem 9.9]{book2023}.
\begin{corollary}  \label{c1}
   Let $X$ be a second
countable locally compact Abelian group containing no elements of order $2$.      Let $\alpha$ be a topological automorphism of the group    $X$ satisfying   condition $(\ref{d1})$. Let
  $\xi_1$ and  $\xi_2$ be independent random variables with values in the group
       $X$  and distributions $\mu_1$ and $\mu_2$ with nonvanishing characteristic functions.
   Then the symmetry of  the conditional
distribution of the linear form  $L_2 = \xi_1 + \alpha\xi_2$ given $L_1 = \xi_1 +
\xi_2$ implies that  $\mu_j\in \Gamma(X)$, $j=1, 2$.
\end{corollary}

{\rm Let  $X$ be a second
countable  locally compact Abelian group, let $G$ be the subgroup of $X$ generated by all elements of $X$ of order 2, and let $\alpha$ be a topological automorphism of the group  $X$.
  Suppose that $K={\rm Ker}(I+\alpha)\ne\{0\}$. Let $\xi_1$ and $\xi_2$ be independent identically distributed random variables with values in the subgroup  $K$ and distribution $\mu$.
 It is obvious that $\alpha x=-x$ for all
 $x\in K$. It is easy to see that the conditional distribution of the linear 
 form $P_2 = \xi_1 - \xi_2$ given
 $P_1 = \xi_1 + \xi_2$ is symmetric.
 Hence  if we consider independent random variables
  $\xi_1$ and $\xi_2$ as independent random variables taking values in $X$, then
the conditional distribution of the linear form   $L_2 = \xi_1 +
\alpha\xi_2$ given $L_1 = \xi_1 + \xi_2$ is also symmetric.
 Since $\mu$ is an arbitrary distribution, from what has been said it follows that 
a necessary condition for a topological automorphism $\alpha$ 
 for Theorem \ref{th1} to be true is an inclusion
\begin{equation}\label{9a}
   {\rm Ker}(I+\alpha)\subset G.
\end{equation}
However, generally speaking, this condition is  not sufficient.   The corresponding example can be constructed  in the case, when the group $X$ is an ${\bm a}$-adic solenoid.

Recall the definition of an ${\bm a}$-adic solenoid.
Let ${\bm a}= (a_0, a_1,\dots, a_n,\dots,)$, where all $a_j \in {\mathbb{Z}}$ and
$a_j > 1$.  Let $\Delta_{\bm a}$ be the group of ${\bm a}$-adic integers 
and let $B$ be the subgroup of the group 
 $\mathbb{R}\times\Delta_{\bm a}$ of the form
$B=\{(n,n{u})\}_{n=-\infty}^{\infty}$, where
${u}=(1, 0,\dots,0,\dots)$. The factor-group $\Sigma_{\bm a}=(\mathbb{R}\times\Delta_{\bm a})/B$ is called an   ${\bm a}$-{adic solenoid} (see e.g. \cite[\S 10]{Hewitt-Ross}). The group 
$\Sigma_{\bm a}$ is compact, connected, has dimension 1, and 
contains no subgroups topologically isomorphic to the circle group  $\mathbb{T}$. Moreover, $\Sigma_{\bm a}$ can contain at most one element of order 2.  Assume that the group $\Sigma_{\bm a}$ 
contains an element  of order 2 and denote by $G$ the subgroup of $\Sigma_\text{\boldmath $a$}$ generated by this element.
It follows from the results proved in \cite{Fe2020}, see also \cite[Proposition 11.19, Theorem 11.20 and Remark 11.22]{book2023})}    that the following statement holds.
\begin{proposition}\label{pr2} Let $\Sigma_{\bm a}$ be an   ${\bm a}$-adic solenoid 
  containing an element of order $2$. Let  $\alpha$  be  a topological automorphism  of the group $\Sigma_{\bm a}$ such that  
  \begin{equation}\label{n9a}
    {\rm Ker}(I+\alpha)=G
\end{equation}
   and hence  condition $(\ref{9a})$ holds. Then there exist
   independent random variables $\xi_1$ and $\xi_2$   with values in the group  $\Sigma_{\bm a}$ and distributions  $\mu_1$ and $\mu_2$ with nonvanishing 
   characteristic functions such that  the conditional distribution of 
   the linear form
  $L_2 = \xi_1 + \alpha\xi_2$ given $L_1 = \xi_1 +
\xi_2$ is symmetric, whereas 
$\mu_j\notin \Gamma(\Sigma_{\bm a})*{\rm M^1}(G)$, $j=1, 2$.
\end{proposition}
An example of   an ${\bm a}$-adic solenoid $\Sigma_{\bm a}$
 and a topological automorphism $\alpha\in{\rm Aut}(\Sigma_{\bm a})$ such that
 (\ref{n9a}) holds
is $\Sigma_{\bm a}$, where  ${\bm a}=(3,3, \dots, 3, \dots)$,
and $\alpha$ is the multiplication by $-3$.

\section{Theorem A is false for the 2-dimensional torus $\mathbb{T}^2$}

We prove in this section that  Theorem  A is false for
  the 2-dimensional torus   $\mathbb{T}^2$.
Let $X=\mathbb{T}^2$. Denote by $x=(z, w)$, $z, w\in
\mathbb{T}$, elements of the group $X$. The character group $Y$ of the group $X$ is topologically isomorphic to the group    $\mathbb{Z}^2$.   Denote by $y=(m,n)$, $m, n\in\mathbb{Z}$, elements of the group $Y$.  Every automorphism $\alpha\in {\rm Aut}(X)$
is defined by an integer-valued matrix
$\left(\begin{matrix}a&b\\
c&d\end{matrix}\right)$, where $|ad-bc|=1$  and $\alpha$ acts on
$X$ as follows
$$
\alpha(z,w)=(z^a w^c,z^b w^d),\quad (z,w)\in X.
$$
The adjoint automorphism    $\widetilde\alpha\in {\rm
Aut}(Y)$ is of the form
$$\widetilde\alpha(m,n)=(am+bn,cm+dn),\quad (m,n)\in Y.$$
We identify    $\alpha$    with the   matrix
$\left(\begin{matrix}a&b\\
c&d\end{matrix}\right)$  and $\widetilde\alpha$  with the   matrix
$\left(\begin{matrix}a&c\\
b&d\end{matrix}\right)$. 

It follows from the definition of the Gaussian distribution on a 
locally compact Abelian group that the characteristic function of a 
symmetric Gaussian distribution on the group  $X$ is of the form
$$
\hat\gamma(m, n)=\exp\{-\langle A(m, n), (m, n)\rangle\},\quad
(m, n)\in Y,
$$
where $\langle\cdot,\cdot\rangle$ is the standard scalar product in  $\mathbb{R}^2$, $A=(a_{ij})_{i,j=1}^2$ is a symmetric positive semidefinite matrix.

\begin{theorem}\label{th2}  Let  $X=\mathbb{T}^2$   and let  $G$ be the 
subgroup of $X$ generated by all elements 
of $X$ of order
  $2$. Then there exist a topological automorphism  $\alpha$ of the group  
  $X$ satisfying  condition $(\ref{d1})$ and independent random variables   $\xi_1$ and 
$\xi_2$ with values in   $X$ and distributions
  $\mu_1$ and $\mu_2$ with nonvanishing characteristic functions such that the conditional distribution of the linear form
$L_2=\xi_1+\alpha\xi_2$ given $L_1=\xi_1+\xi_2$
is symmetric,  whereas $\mu_j\notin \Gamma(X)*{\rm M^1}(G)$, $j=1, 2$.
\end{theorem}
 
 For the proof of Theorem \ref{th2} we need the following lemmas.
\begin{lemma}\label{lem7}    {\rm (\!\!\cite{MiFe4}, see also \cite[Lemma 11.2]{Fe5})}   Consider the $2$-dimensional torus  $\mathbb{T}^2$. Let $\alpha=\left(\begin{matrix}a&b\\
c&d\end{matrix}\right)$ be a topological automorphism of the group $\mathbb{T}^2$.
If $ad-bc=1$ and   $a+d<-2$, then there exist symmetric positive semidefinite 
$2\times 2$ matrices $A_1$ and $A_2$ such that
\begin{equation}\label{03.05.2}
\det A_1=\det A_2 > 0
\end{equation}
and
\begin{equation}\label{03.05.3}
A_1+A_2\widetilde\alpha =0.
\end{equation}
\end{lemma}
\begin{lemma}\label{lem6}   {\rm (\!\!\cite[Lemma 16.1]{Fe5})}  Let $X$ be a second countable locally compact Abelian group   with character group $Y$. 
Let $\alpha$ be a topological automorphism of  $X$.
Let $\xi_1$ and  $\xi_2$ be independent random variables with values in the group  $X$ and distributions $\mu_1$ and $\mu_2$. The conditional distribution of the linear form  $L_2 = \xi_1 + \alpha\xi_2$ given $L_1 = \xi_1 + \xi_2$ is symmetric if and only if the characteristic functions
 $\hat\mu_j(y)$  satisfy the equation
\begin{equation}\label{03.05.18}
\hat\mu_1(u+v )\hat\mu_2(u+\widetilde\alpha v )=
\hat\mu_1(u-v )\hat\mu_2(u-\widetilde\alpha v), \ \ u, v \in Y.
\end{equation}
\end{lemma}
\noindent\textbf{\textit{Proof of Theorem $\bf{\ref{th2}}$}}  Consider a topological automorphism $\alpha=\left(\begin{matrix}a&b\\
c&d\end{matrix}\right)$ of the group $X$ such that  $\det\alpha=1$ and $a+d=-3$. 
It follows from this that $\det(I+\alpha)=-1$ and hence $I+\alpha\in{\rm Aut}(X)$.
For this reason condition (\ref{d1}) holds.
By Lemma \ref{lem7},  there exist symmetric positive semidefinite $2\times 2$ matrices $A_1$ and $A_2$ such that   (\ref{03.05.2}) and (\ref{03.05.3}) are valid.

Note that $\det(I-\alpha)=\det(I-\widetilde\alpha)=5$. 
Hence $I-\widetilde\alpha\notin{\rm Aut}(X)$ and it follows from this that 
$H=(I-\widetilde\alpha)(Y)$ is a proper subgroup of $Y$. This implies that
 $K=A(X, H)\ne \{0\}$.

  Let us check that $Y^{(2)}\setminus H\ne\emptyset$. Suppose the contrary is true, i.e., $Y^{(2)}\subset H$. Then, on the one hand, for $H$ we have the following possibilities:
  \begin{equation}\label{16.1}
H=\{(2m, 2n):m, n\in \mathbb{Z}\}, \quad H=\{(2m, n):m, n\in \mathbb{Z}\},
\end{equation}
\begin{equation}\label{n16.1}
H=\{(m, 2n):m, n\in \mathbb{Z}\}, \quad H=\{(2m, 2n),(2m-1, 2n-1):m, n\in \mathbb{Z}\}.
\end{equation}
On the other hand,  $H=\{((1-a)m-bn,-cm+(1-d)n):m, n\in \mathbb{Z}\}$  
and this implies that
\begin{equation}\label{16.2}
(-b, 1-d), (1-a, -c)\in H.
\end{equation}
Since $a+d=-3$, this implies that $a$ and $d$ have different parity.
In view  of  $ad-bc=1$, we have either $a$ is even  and $b, c, d$ are odd  or
$d$ is even  and $a, b, c$ are odd. Anyway (\ref{16.2}) contradicts (\ref{16.1}) and (\ref{n16.1}).
Hence  $Y^{(2)}\setminus H\ne\emptyset$.

Take $0< \kappa < 1$ and consider on the group $X$ the distribution
  $$\pi_1=\kappa E_{(1,1)}+(1-\kappa)m_K$$ and the signed measure
$$\pi_2=\frac{1}{\kappa}E_{(1,1)}+\frac{\kappa-1}{\kappa}m_K.$$  Since $H=A(Y, K)$, it follows from (\ref{10.1}) that
$$
\widehat m_K(y)=
\begin{cases}
1, & \text{\ if\ }\   y\in H,
\\ 0, & \text{\ if\ }\ y\not\in H,
\end{cases}
$$
and the characteristic functions  $\hat\pi_j(y)$ are of the form
\begin{equation}\label{04.05.1}
\hat\pi_1(y)=
\begin{cases}
1, & \text{\ if\ }\   y\in H,
\\  \kappa, & \text{\ if\ }\ y\not\in H,
\end{cases}\quad\quad \hat\pi_2(y)=
\begin{cases}
1, & \text{\ if\ }\   y\in H,
\\  \displaystyle{\frac{1}{\kappa}}, & \text{\ if\ }\ y\not\in H.
\end{cases}
\end{equation}
Let us check that the characteristic functions    $\hat\pi_j(y)$ satisfy equation
(\ref{03.05.18}). To see this we verify that for any $u, v \in Y$    both sides of equation
(\ref{03.05.18}) are equal to 1. Suppose that there exist  $u, v\in Y$ such that
$\hat\pi_1(u+v)\hat\pi_2(u+\widetilde\alpha v)\neq 1$. In view of (\ref{04.05.1}) then either
  $u+v\in H, u+\widetilde\alpha v\not\in H$ or $u+v\not\in H,
u+\widetilde\alpha v\in H$. In both cases we obtain that
$(I-\widetilde\alpha)v\not\in H$. This is impossible because $H=(I-\widetilde\alpha)(Y)$. Thus, the left hand-side of equation  (\ref{03.05.18}) for any 
$u, v \in Y$ is equal to 1. Reasoning similarly we check that the right hand-side of equation (\ref{03.05.18}) for any $u, v \in Y$ is also equal 1. Hence  the characteristic functions    $\hat\pi_j(y)$ satisfy equation
(\ref{03.05.18}).

It is easy to see that since (\ref{03.05.3}) is valid, the functions 
$$
h_j(m, n)=\exp\{-\langle A_j(m, n), (m, n)\rangle\}, \quad  j=1, 2,
$$ 
satisfy equation (\ref{03.05.18}).  Hence  the functions
\begin{equation}\label{04.05.2}
g_j(m, n)=h_j^k(m, n)\hat\pi_j(m, n), \quad (m, n)\in Y,   \ j=1, 2,
\end{equation}
for any natural $k$  also satisfy equation (\ref{03.05.18}).  

It follows from (\ref{03.05.2}) that there is $\varepsilon>0$ such that
\begin{equation}\label{n08.07.1}
\langle A_j(m, n), (m, n)\rangle\ge\varepsilon(m^2+n^2), \quad  j=1, 2.
\end{equation}
Since $g_j(0, 0)=1$,   inequalities (\ref{n08.07.1}) imply 
that for a big enough $k$  the inequalities
\begin{equation}\label{08.07.1}
\sum_{(m, n)\in Y} g_j(m, n) < 2, \quad j=1, 2,
\end{equation}
hold  true. Put
$$
\rho_j(z, w)=\sum_{(m, n)\in Y}
g_j(m,n)\bar{z}^m\bar{w}^n, \quad (z,w) \in X,  \   j=1, 2.
$$
Since  $g_j(-y)=g_j(y)$ for all $y\in Y$, it follows from (\ref{08.07.1})
that then
$\rho_j(z, w)>0$ for all $(z,w) \in X$,  $j=1, 2$.
It is also obvious that
$$
\int\limits_{X}\rho_j(z, w)dm_{X}(z, w)=1, \quad j=1, 2.
$$
Thus, the functions $\rho_j(z, w)$ are densities  with respect to 
$m_{X}$ of   some distributions   $\mu_j$   on the group $X$. In so doing,
$\hat\mu_j(y)=g_j(y)$, $j=1, 2$. Let $\xi_1$
and $\xi_2$ are independent random variables with values  in the group
$X$ and distributions $\mu_1$ and $\mu_2$. Since the characteristic functions  $\hat\mu_j(y)$ satisfy equation
(\ref{03.05.18}), by Lemma \ref{lem6}, the conditional distribution of the linear form
  $L_2 = \xi_1 + \alpha\xi_2$ given $L_1 = \xi_1 + \xi_2$ is symmetric.
 Note now that $Y^{(2)}=A(Y, G)$. For this reason if   $\mu\in{\rm M}^1(X)$ and $\mu\in\Gamma(X)*{\rm M}^1(G)$, then the restriction of the characteristic function $\hat\mu(y)$
to the subgroup $Y^{(2)}$ is the characteristic function of a Gaussian distribution.
Since $Y^{(2)}\setminus H\ne\emptyset$, it follows from (\ref{04.05.1}) and (\ref{04.05.2}) that $\mu_j\notin\Gamma(X)*{\rm M}^1(G)$, $j=1, 2$. \hfill$\Box$

\medskip

     In view of Theorem \ref{th2}, it is interesting to note that  Theorem A holds true for the group ${\mathbb
T}^2$  if we substitute condition $(\ref{d1})$ for    $(\ref{fe18})$ (\!\!\cite{Fe4}, see also \cite[Theorem 16.8]{Fe5}). Namely, the following statement holds.
\begin{b*} \label{b*} Consider the $2$-dimensional torus  $\mathbb{T}^2$   
and let  $G$ be the subgroup of $\mathbb{T}^2$ generated by all 
elements of $\mathbb{T}^2$ of order
  $2$. Assume that a topological automorphism  $\alpha$ of the group  $\mathbb{T}^2$ satisfies  conditions $(\ref{fe18})$. Let $\xi_1$ and $\xi_2$  be independent random variables   with values in the group   $\mathbb{T}^2$ and distributions
  $\mu_1$ and $\mu_2$ with nonvanishing characteristic functions. 
  If the conditional distribution of the linear form
$L_2=\xi_1+\alpha\xi_2$ given $L_1=\xi_1+\xi_2$
is symmetric, then 
$\mu_j = \gamma_j
* \omega_j$,  where $\gamma_j\in \Gamma(\mathbb{T}^2)$,
 $\omega_j\in{\rm M^1}(G)$, $j=1, 2$.
 In so doing the Gaussian distributions
$\gamma_j$ are concentrated on cosets of the same dense one-parameter subgroup in 
$\mathbb{T}^2$.
\end{b*}

In conclusion, taking into account Theorems \ref{th1} and \ref{th2}, 
 we formulate the following conjecture.

\begin{conjecture}  Let $X$ be a second
countable locally compact Abelian group, let $G$ be the subgroup of $X$
generated by all elements of $X$ of order $2$, and let $\alpha$ be a topological automorphism of the group    $X$ satisfying   condition $(\ref{d1})$. Let
  $\xi_1$ and  $\xi_2$ be independent random variables with values in  
       $X$  and distributions $\mu_1$ and $\mu_2$ with nonvanishing characteristic functions.
   The symmetry of  the conditional
distribution of the linear form  $L_2 = \xi_1 + \alpha\xi_2$ given $L_1 = \xi_1 +
\xi_2$ implies that $\mu_j\in \Gamma(X)*{\rm M^1}(G)$, $j=1, 2$, 
 if and only if the group $X$ contains no  subgroups topologically
isomorphic to the $2$-dimensional torus   $\mathbb{T}^2$.
\end{conjecture}

The sufficiency in this assertion follows from  Theorem \ref{th1}.

\medskip
 
\textbf{Acknowledgements}
This article was written during my stay at 
the Department of Mathematics University of Toronto as a Visiting Professor. 
I am very grateful to Ilia Binder for his invitation and support.
I would also like to thank the reviewer  for carefully reading the article
and useful comments.

\noindent 
B. Verkin Institute for Low Temperature Physics and Engineering\\
of the National Academy of Sciences of Ukraine\\
Nauky Ave. 47,
61103 Kharkiv, Ukraine

\bigskip

\noindent Department of Mathematics  
University of Toronto \\
40 St. George Street
Toronto, ON,  M5S 2E4
Canada 

\medskip

\noindent e-mail:    gennadiy\_f@yahoo.co.uk

\end{document}